\documentclass[a4paper,12pt]{article}

\usepackage{latexsym}
\usepackage{mathrsfs}
\usepackage{amsfonts,amsmath,amssymb}
\usepackage{indentfirst}
\usepackage{subeqnarray}
\usepackage{verbatim}
\usepackage[pdftex]{graphicx}
\usepackage{epstopdf}
\usepackage{booktabs}

\newcommand{\bx}{\mbox{\boldmath{$x$}}}

\newcommand{\bz}{\mbox{\boldmath{$z$}}}

\newcommand{\bzero}{\mbox{\boldmath{$0$}}}

\newcommand{\bI}{\mbox{\boldmath{$I$}}}

\newcommand{\fb}{\mbox{\boldmath{$f$}}}
\newcommand{\bg}{\mbox{\boldmath{$g$}}}

\newcommand{\bS}{\mbox{\boldmath{$S$}}}
\newcommand{\bu}{\mbox{\boldmath{$u$}}}

\newcommand{\bv}{\mbox{\boldmath{$v$}}}
\newcommand{\bV}{\mbox{\boldmath{$V$}}}
\newcommand{\bH}{\mbox{\boldmath{$H$}}}

\newcommand{\beeta}{\mbox{\boldmath{$\eta$}}}
\newcommand{\bzeta}{\mbox{\boldmath{$\zeta$}}}

\newcommand{\bvarepsilon}{\mbox{\boldmath{$\varepsilon$}}}
\newcommand{\bsigma}{\mbox{\boldmath{$\sigma$}}}

\newcommand{\bnu}{\mbox{\boldmath{$\nu$}}}

\newcommand{\bxi}{\mbox{\boldmath{$\xi$}}}

\newcommand{\cA}{\mbox{${\cal A}$}}

\newcommand{\cH}{\mbox{${\cal H}$}}
\newcommand{\cL}{\mbox{${\cal L}$}}
\newcommand{\cQ}{\mbox{${\cal Q}$}}
\newcommand{\cS}{\mbox{${\cal S}$}}

\newcommand{\cV}{\mbox{${\cal V}$}}
\newcommand{\cW}{\mbox{${\cal W}$}}

\newcommand{\weak}{\rightharpoonup}

\newcommand{\bVV}{\mbox{\scriptsize ${\boldsymbol V}^*\times {\boldsymbol V}$}}
\newcommand{\bVVdiv}{\mbox{\scriptsize ${\boldsymbol V}^*_{\rm div}\times {\boldsymbol V}_{\rm div}$}}

\newcommand{\cVVdiv}{\mbox{\scriptsize ${\cal V}^*_{\rm div}\times {\cal V}_{\rm div}$}}

\newtheorem{theorem}{Theorem}[section]
\newtheorem{lemma}[theorem]{Lemma}

\newtheorem{problem}[theorem]{Problem}
\newtheorem{proposition}[theorem]{Proposition}
\newtheorem{remark}[theorem]{Remark}
\numberwithin{equation}{section}
\newenvironment{proof}[1][Proof]{\textbf{#1.} }
{\ \rule{0.75em}{0.75em}\smallskip}

\usepackage[pdftex,unicode,colorlinks,linkcolor=blue]{hyperref}

\begin{document}

\begin{center}
\Large\bf On Well-posedness of a\\
Nonstationary Stokes Hemivariational Inequality
\end{center}

\begin{center}
\large
Weimin Han\footnote{Department of Mathematics, University of Iowa, Iowa City, IA 52242-1410, USA.
E-mail: {\tt weimin-han@uiowa.edu}.  The work of this author was partially supported by Simons Foundation Collaboration Grants, No.\ 850737.}\quad and \quad
Shengda Zeng\footnote{National Center for Applied Mathematics in Chongqing, and School of Mathematical Sciences, Chongqing Normal University, Chongqing 401331, China. E-mail: {\tt zengshengda@163.com}. The work of this author was supported by the National Natural Science Foundation of China under Grant No.\ 12371312, and the Natural Science Foundation of Chongqing under Grant No.\ CSTB2024NSCQ-JQX0033.}
\end{center}

\medskip

\begin{center}
\large \emph{Dedicated to Professor Zuhair Nashed on the occasion of his 90th birthday}
\end{center}

\medskip

\begin{quote}
{\bf Abstract.}  This paper is devoted to the well-posedness analysis of a nonstationary Stokes
hemivariational inequality for an incompressible fluid flow described by the Stokes equations subject to 
a nonsmooth boundary condition of friction type described by the Clarke subdifferential.  In a recent 
paper \cite{HYZ26}, well-posedness of the nonstationary Stokes hemivariational inequality is studied 
for both the velocity and pressure fields.  The solution existence is shown through 
a limiting procedure based on temporally semi-discrete approximations for both the velocity and pressure 
fields.  In this paper, a refined well-posedness analysis is provided on the nonstationary Stokes
hemivariational inequality under more natural assumptions on the problem data.  
The solution existence is first shown for the velocity field through a limiting 
procedure based on temporally semi-discrete approximations of a reduced problem and then the pressure 
field is recovered with the help of an inf-sup property.  In this way, assumptions on the source term and the 
initial velocity needed in \cite{HYZ26} are weakened, and a compatibility condition on 
initial values of the data is dropped.  Moreover, several hemivariational 
inequalities are introduced for the mathematical model and their equivalence is explored. 
\end{quote}

\medskip

\section{Introduction}

Since the pioneering work of Panagiotopoulos (\cite{Pa83}), there has been extensive research on modeling, 
analysis, numerical solution and applications of hemivariational inequalities and the more general
variational-hemivariational inequalities. Several research monographs are available for detailed expositions
of mathematical theories and applications of variational-hemivariational inequalities, e.g., just to mention a few
recent ones, \cite{CL2021, Han2024, MOS2013, NP1995, SM2025}.  In most references in the area, 
abstract surjectivity results on pseudomonotone operators are needed to investigate the well-posedness of 
variational-hemivariational inequalities. For an approach without the need of the notion and results of 
pseudomonotone operators, the reader is referred to the book \cite{Han2024} for an accessible coverage 
of the mathematical theory for variational-hemivariational inequalities, based on the techniques developed
in \cite{Han20, Han21}.  Since there are no analytic solution formulas for variational-hemivariational 
inequalities from applications, numerical methods are needed to solve the problems. 
An early comprehensive reference on the finite element method for solving hemivariational 
inequalities is \cite{HMP1999}.  There has been substantial progress on numerical analysis of 
variational-hemivariational inequalities in the recent years, and the reader is referred to
the two survey papers \cite{HS19AN, HFWH25} and the references therein.  See also the
survey paper \cite{Han26} for an elementary level overview of mathematical theory and numerical 
analysis of stationary variational-hemivariational inequalities for beginners of the area.

Nonstationary Stokes variational inequalities are used to describe incompressible fluid flow problems 
involving a non-leak nonsmooth monotone slip condition of friction type. They have been studied in 
a number of papers, e.g., \cite{Dj14,Fu02,LL08}. When a relevant superpotential function in these papers
is no longer assumed to be convex, a hemivariational inequality arises. This paper is devoted to the 
study of a nonstationary Stokes hemivariational inequality, i.e.,
a hemivariational inequality of the time-dependent Stokes equations for an incompressible fluid
flow subject to a non-leak nonsmooth non-monotone slip condition of friction type.  
In the recent paper \cite{HYZ26}, the well-posedness of the nonstationary Stokes hemivariational
inequality is proved under a compatibility condition on initial values of the data. The solution 
existence is shown through a limiting procedure based on temporally semi-discrete approximations
for both the velocity and pressure fields.  In this paper, the solution existence is first shown 
for the velocity field through a limiting procedure based on temporally semi-discrete approximations
of a reduced problem and then the pressure field is recovered with the help of an inf-sup property.
In this way, the compatibility condition on initial values of the data needed in \cite{HYZ26} is dropped. 
We note in passing that in \cite{HYZ26}, numerical analysis is extended to the nonstationary Stokes 
hemivariational inequality.  The Rothe method has been used in studying 
a variety of time-dependent problems, see for example \cite{Rou2013, Ka13, BCKYZ15, ZMN22}.  
Numerical methods have been studied in a number of papers for solving stationary mixed 
hemivariational inequalities in fluid mechanics, e.g., \cite{FCHCD20, LHZ22} for a stationary Stokes 
hemivariational inequality, \cite{HJY23} for a stationary Navier-Stokes hemivariational inequality. 

The organization of this paper is as follows. In Section \ref{sec:pre} we present some definitions
and auxiliary material. Several formulations of the problem are introduced and their equivalence is 
explored in Section \ref{sec:var}; assumptions on the data are also stated in the section.
In Section \ref{sec:reduced}, reduced formulations of the problem for the velocity variable are presented
in which the pressure variable is eliminated.  In Section \ref{sec:exi},  the existence and uniqueness 
of the velocity variable are investigated for the reduced formulations.  In Section \ref{sec:p}, 
the pressure component of the solution for the problem is recovered.  In Section \ref{sec:c}, Lipschitz 
continuous dependence of the velocity solution on the source function and the initial value is shown.
In Section \ref{sec:conc}, we provide a concluding remark on the main contribution of this paper.

\section{Preliminaries}\label{sec:pre}

Let $V$ be a Banach space with the norm $\|\cdot\|_V$.  Denote by $V^*$ its topological dual, and by 
$\langle\cdot,\cdot\rangle_{V^*\times V}$ the duality pairing between $V^*$ and $V$.  Weak convergence
will be indicated by the symbol $\weak$.  The symbol $2^{V^*}$ represents the set of all subsets of $V^*$. 

We first recall the notions of the generalized directional derivative and the generalized subdifferential
of a locally Lipschitz function.  These notions are needed in expressing hemivariational inequalities.
Let $\Psi\colon V\to \mathbb{R}$ be a locally Lipschitz function.  Then the generalized directional derivative 
(in the sense of Clarke) of $\Psi$ at $u\in V$ in the direction $v\in V$ is defined by
\[\Psi^0(u;v)=\limsup_{w\rightarrow u,\lambda\rightarrow0}\frac{\Psi(w+\lambda v)-\Psi(w)}{\lambda},\]
and the generalized subdifferential (in the sense of Clarke) of $\Psi$ at $u$ is defined by
\[ \partial\Psi(u)=\{\xi\in V^*\mid \Psi^0(u;v)\ge\langle\xi,v\rangle_{V^*\times V}\ \forall\,v\in V\}.\]
For properties and detailed discussions of the generalized directional derivative and 
the generalized subdifferential of a locally Lipschitz function, the reader is referred to 
\cite{Cl1983, CLSW1998} or \cite[Section 3.2]{MOS2013}.  

Given a finite time interval $I=(0,T)$, for $1\leq p\leq\infty$, we will use the Bochner 
space $L^p(I;V)$.  For $1\leq q<\infty$, we denote by $BV^q(I;V)$ the space of functions of bounded  
total variation on $I$ of exponent $q$.  This space consists of functions $v\colon\overline{I}\to V$ such that 
\[ \|v\|_{BV^q(I;V)}^q=\sup_{\pi\in\mathcal{P}}\sum_{i=1}^{n(\pi)}\|v(a_i)-v(a_{i-1})\|_V^q<\infty,\]
where $\mathcal{P}$ is the set of all finite partitions $\pi$ of $\overline{I}=[0,T]$:
$0=a_0<a_1<\cdots<a_{n(\pi)}=T$, $n(\pi)$ being a positive integer.

Given Banach spaces $V_1$ and $V_2$ such that $V_1\subset V_2$, define
\[ M^{p,q}(I;V_1,V_2)=L^p(I;V_1)\cap BV^q(I;V_2)\]
which is a Banach space for $1\leq p,q<\infty$ with the norm given by
$\|\cdot\|_{L^p(I;V_1)}+\|\cdot\|_{BV^q(I;V_2)}$.

The following important result is proved in \cite{Ka13}.

\begin{theorem}\label{th:lb}
Let $1\leq p,q<\infty$. Let $V_1\subset V_2\subset V_3$ be real Banach spaces such that $V_1$
is reflexive, the embedding $V_1\subset V_2$ is compact and the embedding $V_2\subset V_3$ is continuous.
Then a bounded subset of $M^{p,q}(I;V_1,V_3)$ is relatively compact in $L^p(I;V_2)$.
\end{theorem}

We will apply the modified Cauchy-Schwarz inequality on several occasions in Section \ref{sec:exi}, with a proper choice of $\epsilon>0$,
\begin{equation}
 a\,b\le \epsilon\,a^2+ c_\epsilon b^2\quad\forall\,a,b\in\mathbb{R},\, \epsilon>0, \ c_\epsilon=1/(4\,\epsilon).
\label{mCS}
\end{equation}

The following version of Aubin-Cellina convergence theorem (\cite[p.\ 60]{AC1984}) will be used.

\begin{theorem}\label{th:sv}
Let $F: X\rightarrow 2^Y$ be a set-valued function between two Banach spaces $X$ and $Y$.  Assume $F$ 
is upper semicontinuous from $X$ into $Y_w$ and the values of $F$ are nonempty, closed and convex subsets 
of $Y$. Let $x_n\colon (0,T)\rightarrow X$, $y_n\colon (0,T)\rightarrow Y$, $n\in \mathbb{N}$, be measurable 
functions such that $x_n$ converges almost everywhere on $(0,T)$ to a function $x\colon (0,T)\rightarrow X$ 
and $y_n$ converges weakly in $L^1(0,T;Y)$ to $y\colon (0,T)\rightarrow Y$. If $y_n(t)\in F(x_n(t))$ for all 
$n\in \mathbb{N}$ and almost all $t\in(0,T)$, then $y(t)\in F(x(t))$ for a.e.\ $t\in(0,T)$.
\end{theorem}

\section{Mathematical model}\label{sec:var}

Let $d=2$ or 3 be the dimension of the spatial domain for the problem under consideration.
Denote the Euclidean norm in $\mathbb{R}^d$ by $|\cdot|$.  Denote by $\mathbb{S}^d$ the space of
symmetric tensors on $\mathbb{R}^d$, or equivalently, the space of symmetric matrices of order $d$. 

Let $\Omega\subset\mathbb{R}^d$ be a Lipschitz domain.  Its boundary
$\Gamma$ is split into two non-trivial non-overlapping parts: $\Gamma=\Gamma_D\cup\Gamma_S$ 
with $\Gamma_D$ relatively closed and $\Gamma_S$ relatively open. 
Let $[0,T]$ be the time interval for a given $T>0$.  Since the boundary is Lipschitz continuous, 
the unit outward normal vector $\bnu=(\nu_1,\cdots,\nu_d)^T$ exists a.e.\ on $\Gamma$.  For a
vector $\bv$, we denote by $v_\nu=\bv\cdot \bnu$ its normal component and 
$\bv_\tau=\bv-v_\nu\bnu$ its tangential component on $\Gamma$.

We consider the motion of an incompressible fluid flow in $\Omega$, modeled by the Stokes equations.
The unknown variables are the velocity field $\bu$ and the pressure field $p$ of the fluid.
Denote by $\mu>0$ the viscosity of the fluid, by $\fb$ the density of the volume force 
applied to the fluid, and by $\bu_0\colon \Omega\to\mathbb{R}^d$ be the initial velocity.  The 
stress tensor is $\bsigma=2\,\mu\,\bvarepsilon(\bu)-p\,\bI$, where $\bvarepsilon(\bu)=(\nabla \bu+(\nabla\bu)^T)/2$ 
and $\bI$ is the identity tensor.  The normal and tangential components of the stress on the boundary 
are $\sigma_\nu=(\bsigma\bnu)\cdot\bnu$ and $\bsigma_\tau=\bsigma\bnu-\sigma_\nu\bnu$.  
We impose the homogeneous Dirichlet boundary condition on $\Gamma_D$ and a non-smooth, possibly
non-monotone, friction condition with a potential function $\psi$ on $\Gamma_S$.  The potential 
function $\psi$ is assumed to be locally Lipschitz continuous.  
The classical pointwise formulation of the problem is the following:
\begin{align}
\bu^\prime-2\mu\,{\rm div}\bvarepsilon(\bu)+\nabla p & =\fb\quad{\rm in}\ (0,T)\times\Omega, \label{S01}\\
{\rm div}\,\bu & =0\quad{\rm in}\ (0,T)\times\Omega,  \label{S02}\\
\bu & = \bzero\quad{\rm on}\ (0,T)\times\Gamma_D,  \label{S03}\\
u_\nu=0,\ -\bsigma_\tau& \in \partial\psi(\bu_\tau)\quad{\rm on}\ (0,T)\times\Gamma_S,  \label{S04}\\
\bu|_{t=0} & = \bu_0\quad{\rm in}\ \Omega. \label{S05}
\end{align}
In \eqref{S04}, $\partial\psi$ denotes the generalized subdifferential of $\psi$ in the sense of Clarke. 

The problem \eqref{S01}--\eqref{S05} will be studied in weak forms. For this purpose, 
we Introduce the function spaces
\begin{align}
\bV & =\left\{\bv\in H^1(\Omega)^d \mid \bv=\bzero\ {\rm on}\ \Gamma_D,\, v_\nu=0\ {\rm on}\ \Gamma_S\right\},
\label{bV}\\
Q & = L^2_0(\Omega) =\left\{q\in L^2(\Omega)\mid (q,1)_{L^2(\Omega)}=0\right\},
\label{Q}\\
\bH & = L^2(\Omega)^d, \label{bH}
\end{align}
where $(\cdot,\cdot)_{L^2(\Omega)}$ is the canonical inner product in $L^2(\Omega)$, and
\begin{align}
\cV & =L^2(0,T;\bV), \label{cV}\\
\cW & =\left\{\bv\in\cV\mid \bv^\prime \in \cV^*\right\}, \label{cW}\\
\cQ & = L^2(0,T;Q), \label{cQ}\\
\cH & = L^2(0,T;H). \label{cH}
\end{align}
Since $|\Gamma_D|>0$, Korn's inequality holds: there exists a constant
$c>0$, depending only on $\Omega$ and $\Gamma_D$, such that
\[ \|\bv\|_{H^1(\Omega)^d}\le c\,\|\bvarepsilon(\bv)\|_{L^2(\Omega)^{d\times d}}\quad\forall\,\bv\in \bV.\]
See e.g.\ \cite[p.\ 79]{NH1981} for a proof of this inequality. 
Consequently, $\bV$ is a Hilbert space with the inner product
\[ (\bu,\bv)_{\boldsymbol V}:=\int_\Omega \bvarepsilon(\bu):\bvarepsilon(\bv)\,dx\quad\forall\,\bu,\bv\in\bV,\]
and the induced norm $\|\bv\|_{\boldsymbol V}=(\bv,\bv)_{\boldsymbol V}^{1/2}$ is equivalent to the standard 
$H^1(\Omega)^d$-norm over $\bV$.  The spaces $Q$, $\bH$, $\cV$, $\cW$, $\cQ$ and $\cH$ are all 
Hilbert spaces with their standard inner products.  We will also need the following subspace of $\bV$:
\begin{equation}
\bV_0:=H^1_0(\Omega)^d.
\label{SpV0}
\end{equation}
We have the trace inequality
\begin{equation}
\|\bv_{\tau}\|_{L^2(\Gamma_S)^d}\le \lambda_\tau^{-1/2}\|\bv\|_{\boldsymbol V}
\quad\forall\,\bv\in \bV,   \label{trace}
\end{equation}
where $\lambda_\tau>0$ is the smallest eigenvalue of the eigenvalue problem
\[ \bu\in\bV,\quad (\bu,\bv)_{\boldsymbol V}=\lambda\,(\bu_{\tau},\bv_{\tau})_{L^2(\Gamma_S)^d}
\quad\forall\,\bv\in \bV. \]

Define the bilinear forms
\begin{align}
a(\bu,\bv)&=2\,\mu\int_\Omega \bvarepsilon(\bu):\bvarepsilon(\bv)\,dx\quad\forall\,\bu,\bv\in\bV, \label{a}\\[1mm]
b(\bv,q) &= \int_\Omega q\,{\rm div}\bv\,dx\quad\forall\,\bv\in\bV,\,q\in Q. \label{b}
\end{align}
Obviously, the bilinear form $a(\cdot,\cdot)\colon \bV\times\bV\to\mathbb{R}$ is symmetric, continuous, and 
$\bV$-elliptic:
\begin{align*}
|a(\bu,\bv)| &\le 2\,\mu\,\|\bu\|_{\boldsymbol V} \|\bv\|_{\boldsymbol V} \quad\forall\,\bu,\bv\in\bV, \\
a(\bv,\bv) &=2\,\mu\,\|\bv\|_{\boldsymbol V}^2 \quad\forall\,\bv\in\bV.
\end{align*}
The bilinear form $b(\cdot,\cdot)\colon \bV\times\ Q\to\mathbb{R}$ is continuous: for a constant $c_b>0$,
\[ |b(\bv,q)| \le c_b \|\bv\|_{\boldsymbol V} \|q\|_Q \quad\forall\,\bv\in\bV,\,q\in Q. \]
Moreover, we have the inf-sup property (\cite{GR1986, Jo2016, Te1979}): for a positive constant $\alpha_b>0$,
\begin{equation}
\sup_{{\boldsymbol 0}\not={\boldsymbol v}\in {\boldsymbol V}_0} \frac{b(\bv,q)}{\|\bv\|_{\boldsymbol V}}
\ge \alpha_b \|q\|_Q\quad\forall\,q\in Q.
\label{inf-sup}
\end{equation}
Define operators $A\colon \bV\to\bV^*$ and $B\colon \bV\to Q^*$ by
\begin{align*}
\langle A\bu,\bv\rangle_{\bVV} & = a(\bu,\bv)\quad\forall\,\bu,\bv\in\bV,  \\
\langle B\bv,q\rangle_{Q^*\times Q} & = b(\bv,q)\quad\forall\,\bv\in\bV,\,q\in Q.
\end{align*}
Then $A\in\cL(\bV;\bV^*)$ is self-adjoint and 
\[ \langle A\bv,\bv\rangle_{\bVV}=2\,\mu\,\|\bv\|_{\boldsymbol V}^2\quad\forall\,\bv\in \bV. \]

Concerning the data, we assume
\begin{equation}
\fb\in \mathcal{V}^*,\quad \bu_0\in \bH,
\label{assump1}
\end{equation}
and\\
\underline{$H(\psi)$}: $\psi\colon\Gamma_S\times(0,T)\times\mathbb{R}^d\to\mathbb{R}$ is such that

(i) $\psi(\cdot,\cdot,\bxi)$ is measurable on $\Gamma_S\times(0,T)$ for all $\bxi\in \mathbb{R}^d$ and there exists
$\bxi_0\in L^2(\Gamma_S)^d$ such that $\psi(\cdot,t,\bxi_0(\cdot))\in L^1(\Gamma_S)$ for a.e.\ $t\in (0,T)$;

(ii) $\psi(\bx,t,\cdot)$ is locally Lipschitz on $\mathbb{R}^d$ for a.e.\ $(\bx,t)\in \Gamma_S\times (0,T)$;

(iii) $|\beeta|\leq c_0(1+|\bxi|)$ for all $\bxi\in \mathbb{R}^d$, $\beeta\in \partial \psi(\bx,t,\bxi)$,
a.e.\ $(\bx,t)\in \Gamma_S\times (0,T)$ with $c_0>0$;

(iv) $(\beeta_1-\beeta_2)(\bxi_1-\bxi_2)\geq-\alpha_\psi|\bxi_1-\bxi_2|^2$ for all $\beeta_i\in\partial\psi(\bx,t,\bxi_i)$,
$\bxi_i\in \mathbb{R}^d$, $i=1,2$, a.e.\ $(\bx,t)\in \Gamma_S\times (0,T)$ with a constant $\alpha_\psi>0$.

\smallskip
For simplicity in writing, we will use the notation $I_{\Gamma_S}(z)$ for the integral of a function $z$ over $\Gamma_S$.

The following result holds (\cite[Lemma 3.1]{HYZ26}).

\begin{lemma}\label{lem:psi}
Assume $H(\psi)$.  Then,
\[ I_{\Gamma_S}(\psi^0(\bv_{1,\tau};\bv_{2,\tau}-\bv_{1,\tau})
+\psi^0(\bv_{2,\tau};\bv_{1,\tau}-\bv_{2,\tau}))\le 
\alpha_\psi \lambda_\tau^{-1}\|\bv_1-\bv_2\|_{\boldsymbol V}^2\quad\forall\,\bv_1,\bv_2\in\bV.\]
\end{lemma}

Through a standard procedure, we get the following weak formulation of the problem \eqref{S01}--\eqref{S05}.

\begin{problem}\label{p1}
Find $\bu\in \cW$ and $p\in\cQ$ such that for a.e.\ $t\in (0,T)$,
\begin{align}
& \langle \bu^\prime(t),\bv\rangle_{\bVV}+a(\bu(t),\bv)-b(\bv,p(t))+I_{\Gamma_S}(\psi^0(\bu_\tau(t);\bv_\tau))
\ge \langle \fb(t),\bv\rangle_{\bVV}\quad \forall\,\bv\in \bV, \label{p1_1}\\[0.5mm]
& b(\bu(t),q)=0\quad\forall\,q\in Q, \label{p1_2}
\end{align}
and 
\begin{equation}
\bu(0)=\bu_0.  \label{p1_3}
\end{equation}
\end{problem}

Note that since $\cW\subset C([0,T];\bH)$, the initial condition \eqref{p1_3} is well-defined.

By allowing $\bv$ and $q$ to depend on $t$ and integrating \eqref{p1_1} and \eqref{p1_2} 
with respect to $t$, we obtain another weak formulation of the problem \eqref{S01}--\eqref{S05}.

\begin{problem}\label{p2}
Find $\bu\in \cW$ and $p\in\cQ$ such that 
\begin{align}
& \int_0^T\left[\langle \bu^\prime(t),\bv(t)\rangle_{\bVV}+a(\bu(t),\bv(t))-b(\bv(t),p(t))
+I_{\Gamma_S}(\psi^0(\bu_\tau(t);\bv_\tau(t)))\right] dt\nonumber\\[0.5mm]
&{}\qquad \ge \int_0^T\langle \fb(t),\bv(t)\rangle_{\bVV} dt\quad \forall\,\bv\in \cV, \label{p2_1}\\[0.5mm]
& \int_0^T b(\bu(t),q(t))\,dt=0\quad\forall\,q\in \cQ, \label{p2_2}
\end{align}
and 
\begin{equation}
\bu(0)=\bu_0.  \label{p2_3}
\end{equation}
\end{problem}

It turns out that the above two problems are equivalent in the sense that a pair of functions 
$(\bu,p)\in\cW\times\cQ$ is a solution of one problem if and only if it is a solution of the other
(\cite[Proposition 3.4]{HYZ26}).

\begin{proposition}\label{prop:equi}
Problem \ref{p1} and Problem \ref{p2} are equivalent.
\end{proposition}

To proceed further, we need function spaces on $\Gamma_S$:
\begin{align}
\bS&=L^2(\Gamma_S)^d, \label{bS}\\
\cS&=L^2(0,T;\bS). \label{cS}
\end{align}
Define a functional
\begin{equation}
\Psi(\bz)=I_{\Gamma_S}(\psi(\bz)),\quad \bz\in\bS.
\label{Psi}
\end{equation}
Denote by $\gamma\colon \bv\mapsto \bv_\tau$ a trace operator over $\bV$.  Note that $\gamma\in{\cal L}(\bV;\bS)$ 
is compact.  Since (cf.\ \cite[Theorem 3.47\,(iv)]{MOS2013})
\[ \Psi^0(\gamma\bv_1;\gamma\bv_2) \le I_{\Gamma_S}(\psi^0(\gamma\bv_1;\gamma\bv_2))\quad\forall\,\bv_1,\bv_2\in\bV,\]
we deduce from Lemma \ref{lem:psi} that
\begin{equation}
\Psi^0(\gamma\bv_1;\gamma\bv_2-\gamma\bv_1)+\Psi^0(\gamma\bv_2;\gamma\bv_1-\gamma\bv_2)
\le \alpha_\psi \lambda_\tau^{-1} \|\bv_1-\bv_2\|_{\boldsymbol V}^2\quad\forall\,\bv_1,\bv_2\in\bV.
\label{3.27a}
\end{equation}

Now introduce an auxiliary problem:

\begin{problem}\label{p3.5}
Find $\bu\in \cW$ and $p\in\cQ$ such that for a.e.\ $t\in (0,T)$,
\begin{align}
& \langle \bu^\prime(t),\bv\rangle_{\bVV}+a(\bu(t),\bv)-b(\bv,p(t))+\Psi^0(\gamma\bu(t);\gamma\bv))
\ge \langle \fb(t),\bv\rangle_{\bVV}\quad \forall\,\bv\in \bV, \label{3.28}\\[0.5mm]
& b(\bu(t),q)=0\quad\forall\,q\in Q, \label{3.29}
\end{align}
and 
\begin{equation}
\bu(0)=\bu_0.  \label{3.30}
\end{equation}
\end{problem}

Similar to Problem \ref{p2}, we have a global version of Problem \ref{p3.5}.

\begin{problem}\label{p3.6}
Find $\bu\in \cW$ and $p\in\cQ$ such that 
\begin{align}
& \int_0^T\left[\langle \bu^\prime(t),\bv(t)\rangle_{\bVV}+a(\bu(t),\bv(t))-b(\bv(t),p(t))
+\Psi^0(\gamma\bu(t);\gamma\bv(t))\right] dt\nonumber\\[0.5mm]
&{}\qquad \ge \int_0^T\langle \fb(t),\bv(t)\rangle_{\bVV}dt\quad \forall\,\bv\in \cV, \label{3.31}\\[0.5mm]
& \int_0^T b(\bu(t),q(t))\,dt=0\quad\forall\,q\in \cQ, \label{3.32}
\end{align}
and 
\begin{equation}
\bu(0)=\bu_0.  \label{3.33}
\end{equation}
\end{problem}

By adapting the proof of Proposition \ref{prop:equi} in \cite{HYZ26}, we can show the next result.

\begin{proposition}\label{prop:equi3.5}
Problem \ref{p3.5} and Problem \ref{p3.6} are equivalent.
\end{proposition}

A relation between the two pairs of problems, Problems \ref{p1} and \ref{p2} vs.\ Problems \ref{p3.5} and \ref{p3.6}, is the following.

\begin{proposition} \label{prop:3.8}
Any solution of Problem \ref{p3.5} is a solution of Problem \ref{p1}.  If Problem \ref{p3.5} has a 
solution and a solution of Problem \ref{p1} is unique, then both problems admit a unique solution and 
the solution is the same to the two problems.
\end{proposition}
\begin{proof}
The first statement follows from the property
\[ \Psi^0(\gamma\bu(t);\gamma\bv))\le I_{\Gamma_S}(\psi^0(\gamma\bu(t);\gamma\bv(t))).\]
The second statement is obvious.  \hfill 
\end{proof}

Starting with Problem \ref{p3.5}, we consider the next inclusion problem.

\begin{problem}\label{p3.9}
Find $\bu\in \cW$ and $p\in\cQ$ such that for a.e.\ $t\in (0,T)$,
\begin{align}
& \bu^\prime(t)+A\bu(t)-B^*p(t)+\gamma^*\partial\Psi(\gamma\bu(t))\ni \fb(t), \label{3.38}\\[0.5mm]
& B\bu(t)=0, \label{3.39}
\end{align}
and 
\begin{equation}
\bu(0)=\bu_0.  \label{3.39a}
\end{equation}
\end{problem}

This problem can be equivalently rewritten as the following.

\begin{problem}\label{p3.10}
Find $\bu\in \cW$, $p\in\cQ$ and $\bxi\in \cS^*$ such that for a.e.\ $t\in (0,T)$,
\begin{align}
& \bu^\prime(t)+A\bu(t)-B^*p(t)+\gamma^*\bxi(t) = \fb(t), \label{3.40}\\[0.5mm]
& \bxi(t) \in \partial\Psi(\gamma\bu(t)),\label{3.41}\\[0.5mm]
& B\bu(t)=0, \label{3.42}
\end{align}
and 
\begin{equation}
\bu(0)=\bu_0.  \label{3.42a}
\end{equation}
\end{problem}

We have the next result.

\begin{proposition}\label{prop:3.11}
If $(\bu,p)\in \cW\times\cQ$ is a solution of Problem \ref{p3.9}, or if $(\bu,p,\bxi)\in \cW\times\cQ
\times \cS^*$ is a solution of Problem \ref{p3.10}, then $(\bu,p)\in \cW\times\cQ$ is a solution of Problem \ref{p3.5}.
\end{proposition}
\begin{proof}
Note that for all $\bxi(t)\in\partial\Psi(\gamma\bu(t))$, we have
\[ \langle\gamma^*\bxi(t),\bv\rangle_{\bVV}\le \Psi^0(\gamma\bu(t);\gamma\bv)\quad\forall\,\bv\in\bV.\]
Hence, \eqref{3.38} implies \eqref{3.28}.  \hfill
\end{proof}

\section{Reduced formulations for the velocity field}\label{sec:reduced}

In this section, we consider reduced forms of the problems from last section by eliminating the 
pressure variable $p$.  For this purpose, we introduce the following subspaces of $\bV$, $\cV$ and $\cW$:
\begin{align}
\bV_{\rm div} & =\left\{\bv\in\bV \mid {\rm div}\bv=0\ {\rm in}\ \Omega\right\},\label{bVdiv}\\
\cV_{\rm div} & =L^2(0,T;\bV_{\rm div}),\label{cVdiv}\\
\cW_{\rm div} & =\left\{\bv\in\cV_{\rm div} \mid \bv^\prime\in \cV_{\rm div}^* \right\}. \label{cWdiv}
\end{align}
Endowed by the inner products and norms in $\bV$, $\cV$ and $\cW$, these subspaces are themselves Hilbert spaces.
Since $\bV_{\rm div}\subset \bV$, we have $\bV^*\subset \bV_{\rm div}^*$ and the duality pairing 
$\langle\cdot,\cdot\rangle_{\bVV}$ is automatically extended to $\bV_{\rm div}^*\times \bV_{\rm div}$ 
with the property
\begin{equation}
\langle \bg,\bv\rangle_{\bVVdiv} = \langle \bg,\bv\rangle_{\bVV}\quad\forall\,\bg\in \bV^*,\,\bv\in \bV_{\rm div}.
\label{4.3a}
\end{equation}
A similar comment holds regarding the duality pairing on $\cV_{\rm div}^*\times \cV_{\rm div}$.  For $\bg\in \bV^*$ and $\bv\in \bV_{\rm div}$,
since $\|\bg\|_{{\boldsymbol V}^*_{\rm div}}\le \|\bg\|_{{\boldsymbol V}^*}$ and 
$\|\bv\|_{{\boldsymbol V}_{\rm div}}=\|\bv\|_{\boldsymbol V}$, we have the inequality
\[ \left| \langle\bg,\bv\rangle_{\bVVdiv} \right| \le \|\bg\|_{{\boldsymbol V}^*}\|\bv\|_{\boldsymbol V}. \]

Denote by $\gamma_{\rm div}$ the restriction of $\gamma\in \cL(\bV;\bS)$ on $\bV_{\rm div}$: 
$\gamma_{\rm div}\in \cL(\bV_{\rm div};\bS)$.  
The reduced forms of Problem \ref{p3.5} and Problem \ref{p3.6} are the following.

\begin{problem}\label{p3.5div}
Find $\bu\in \cW_{\rm div}$ such that for a.e.\ $t\in (0,T)$,
\begin{align}
\langle \bu^\prime(t),\bv\rangle_{\bVVdiv}+a(\bu(t),\bv)+\Psi^0(\gamma_{\rm div}\bu(t);\gamma_{\rm div}\bv)
\ge \langle \fb(t),\bv\rangle_{\bVVdiv}\quad \forall\,\bv\in \bV_{\rm div}, \label{3.51}
\end{align}
and 
\begin{equation}
\bu(0)=\bu_0.  \label{3.51a}
\end{equation}
\end{problem}

\begin{problem}\label{p3.6div}
Find $\bu\in \cW_{\rm div}$ such that 
\begin{align}
& \int_0^T\left[\langle \bu^\prime(t),\bv(t)\rangle_{\bVVdiv}+a(\bu(t),\bv(t))
+\Psi^0(\gamma_{\rm div}\bu(t);\gamma_{\rm div}\bv(t))\right] dt\nonumber\\[0.5mm]
&{}\qquad \ge \int_0^T\langle \fb(t),\bv(t)\rangle_{\bVVdiv}dt\quad \forall\,\bv\in \cV_{\rm div} \label{3.52}
\end{align}
and 
\begin{equation}
\bu(0)=\bu_0.  \label{3.52a}
\end{equation}
\end{problem}

Similar to Proposition \ref{prop:equi3.5}, we can show that Problem \ref{p3.5div} and Problem \ref{p3.6div}
are equivalent.

The counterpart of Problem \ref{p3.9} is

\begin{problem}\label{p3.9div}
Find $\bu\in \cW_{\rm div}$ such that for a.e.\ $t\in (0,T)$,
\begin{align}
\bu^\prime(t)+A\bu(t)+\gamma_{\rm div}^*\partial\Psi(\gamma_{\rm div}\bu(t))\ni \fb(t), \label{3.53}
\end{align}
and 
\begin{equation}
\bu(0)=\bu_0.  \label{3.53a}
\end{equation}
\end{problem}

This problem is interpreted as 

\begin{problem}\label{p3.10div}
Find $\bu\in \cW_{\rm div}$ and $\bxi\in \cS^*$ such that for a.e.\ $t\in (0,T)$,
\begin{align}
& \bu^\prime(t)+A\bu(t)+\gamma_{\rm div}^*\bxi(t) = \fb(t)\quad{\rm in}\ \bV^*_{\rm div}, \label{3.54}\\[0.5mm]
& \bxi(t) \in \partial\Psi(\gamma_{\rm div}\bu(t)),\label{3.55}
\end{align}
and 
\begin{equation}
\bu(0)=\bu_0.  \label{3.54a}
\end{equation}
\end{problem}

Similar to Proposition \ref{prop:3.11}, we claim that if $(\bu,\bxi)\in \cW_{\rm div}\times \cS^*$ is 
a solution of Problem \ref{p3.10div}, then $\bu$ is a solution of Problem \ref{p3.5div}.

\section{Existence and uniqueness for the velocity field}\label{sec:exi}

In this section, we investigate the existence and uniqueness of a solution to Problem \ref{p3.10div}.  
The main effort is on the existence.  For this purpose, we consider a temporally semi-discrete approximation 
of Problem \ref{p3.10div} based on the backward Euler difference for the time derivative.  In the literature,
such an approximation is known as the Rothe method.  For a fixed positive integer $N\in \mathbb{N}$,
define the time step-size $k=T/N$. Denote the nodes $t_n=n\,k$, $0\le n\le N$.
Introduce the piecewise constant interpolant of $\fb$ by
\[ \fb^k_n=\frac{1}{k}\int_{t_{n-1}}^{t_n}\fb(t)\,dt,\quad 1\le n\le N.\]
It is easy to verify that
\[ \|\fb^k_n\|^2_{{\boldsymbol V}^*}\le\frac{1}{k}\int_{t_{n-1}}^{t_n}\|\fb(t)\|^2_{{\boldsymbol V}^*} dt,
\quad 1\le n\le N,\]
and consequently,
\begin{equation}
k\sum_{j=1}^n \|\fb^k_n\|^2_{{\boldsymbol V}^*}\le\int_0^{t_n}\|\fb(t)\|^2_{{\boldsymbol V}^*} dt
\le \|\fb\|^2_{{\cal V}^*}, \quad 1\le n\le N.
\label{5.1a}
\end{equation}

The initial value $\bu_0\in\bH$ can be approximated by a sequence $\{\bu_{k,0}\}\subset \bV$ such that
$\bu_{k,0}\to \bu_0$ in $\bH$ and $\|\bu_{k,0}\|_{\boldsymbol V}\le c\,k^{-1/2}$ for some constant $c>0$,
cf.\ \cite[Theorem 8.9]{Rou2013}.

Then, the semi-discrete approximation of Problem \ref{p3.10div} is the following.

\begin{problem}\label{p3.10divk} 
Find $\bu^k:=\{\bu^k_n\}_{n=0}^N\subset \bV_{\rm div}$ and 
$\bxi^k:=\{\bxi^k_n\}_{n=1}^N\subset \bS^*$ such that for $n=1,\ldots,N$,
\begin{align}
& \frac{1}{k}(\bu^k_n-\bu^k_{n-1})+A\bu^k_n+\gamma_{\rm div}^*\bxi^k_n=\fb^k_n, \label{S1}\\
& \bxi^k_n\in \partial\Psi(\gamma_{\rm div}\bu^k_n), \label{S2}
\end{align}
and 
\begin{equation}
\bu^k_0=\bu_{k,0}.  \label{S0}
\end{equation}
\end{problem}

We assume 
\begin{equation}
m:=2\mu-\alpha_\psi \lambda_\tau^{-1}>0.
\label{small}
\end{equation}
This is called a smallness condition in the literature since the condition sets an upper bound on the
value $\alpha_\psi$ which reflects the strength of the non-convexity of $\psi$.  

The problem \eqref{S1}--\eqref{S2} is equivalent to
\begin{equation}
\frac{1}{k}\langle\bu^k_n-\bu^k_{n-1},\bv\rangle_{\bVVdiv}+a(\bu^k_n,\bv)
+\Psi^0(\gamma_{\rm div}\bu^k_n;\gamma_{\rm div}\bv)
\ge \langle\fb^k_n,\bv\rangle_{\bVVdiv}\quad\forall\,\bv\in \bV_{\rm div}.
\label{3.63}
\end{equation}
The equivalence is due to the facts that
\[ \Psi^0(\gamma_{\rm div}\bu^k_n;\gamma_{\rm div}\bv)\ge 
\langle\bxi^k_n,\gamma_{\rm div}\bv\rangle_{{\boldsymbol S}^*\times {\boldsymbol S}}
\quad\forall\,\bv\in \bV_{\rm div}\]
for any $\bxi^k_n\in \partial\Psi(\gamma_{\rm div}\bu^k_n)$, and given $\bz\in \bS$,
\[ \Psi^0(\gamma_{\rm div}\bu^k_n;\bz)=\max\{\langle\bzeta,\bz\rangle_{{\boldsymbol S}^*\times {\boldsymbol S}}
\mid \bzeta\in \partial\Psi(\gamma_{\rm div}\bu^k_n)\}
=\langle\bxi^k_n,\bz\rangle_{{\boldsymbol S}^*\times {\boldsymbol S}}\]
for one element $\bxi^k_n\in \partial\Psi(\gamma_{\rm div}\bu^k_n)$.  Then,
\[ \Psi^0(\gamma_{\rm div}\bu^k_n;\gamma_{\rm div}\bv)
=\langle\bxi^k_n,\gamma_{\rm div}\bv\rangle_{{\boldsymbol S}^*\times {\boldsymbol S}}
\quad\forall\,\bv\in \bV_{\rm div}.  \]

Rewrite \eqref{3.63} as
\begin{equation}
\langle\bu^k_n,\bv\rangle_{\bVVdiv}+k\,a(\bu^k_n,\bv)+k\,\Psi^0(\gamma_{\rm div}\bu^k_n;\gamma_{\rm div}\bv)
\ge \langle\tilde{\fb}^k_n,\bv\rangle_{\bVVdiv} \quad\forall\,\bv\in \bV_{\rm div}.
\label{3.64}
\end{equation}
where
\[ \tilde{\fb}^k_n=k\,\fb^k_n+\bu^k_{n-1}. \]

By \cite[Theorem 5.30]{Han2024}, if 
\begin{equation}
\alpha_\psi \lambda_\tau^{-1} k<1+2\,\mu\,k,  \label{3.65}
\end{equation}
then \eqref{3.64} has a unique solution $\bu^k_n\in \bV_{\rm div}$, which is the minimizer of the functional
\[ \frac{1}{2}\,\|\bv\|_{\boldsymbol H}^2 + k\,\mu\,\|\bv\|_{\boldsymbol V}^2
+k\,\Psi(\gamma_{\rm div}\bv)-\langle\tilde{\fb}^k_n,\bv\rangle_{\bVVdiv}\]
over $\bV_{\rm div}$.

The condition \eqref{3.65} is guaranteed by the assumption \eqref{small};
alternatively, it is guaranteed if $k<\lambda_\tau/\alpha_\psi$ without the assumption \eqref{small}.

Summarizing, we have the following result on Problem \ref{p3.10divk}.

\begin{theorem}\label{thm:dis1}
Assume \eqref{assump1}, $H(\psi)$ and \eqref{small}.  Then, Problem \ref{p3.10divk} has a solution 
$(\bu^k,\bxi^k)\subset \bV_{\rm div}\times \bS^*$ and $\bu^k$ is unique.
\end{theorem}

\begin{remark}
On uniqueness of $\bxi^k$: The equation \eqref{S1} implies that $\gamma_{\rm div}^*\bxi^k_n$ is unique, 
i.e., the value $\langle\bxi^k_n,\gamma_{\rm div}\bv\rangle_{{\boldsymbol S}^*\times {\boldsymbol S}}
=\langle\bxi^k_n,\bv_\tau\rangle_{{\boldsymbol S}^*\times {\boldsymbol S}}$ is uniquely
defined with respect to $\bv_\tau$ for $\bv\in \bV_{\rm div}$.  In other words, the normal component
of $\bxi^k_n$ does not matter in \eqref{S1} and in terms of the tangential component of $\bxi^k_n$, it is 
unique.  Alternatively, we can choose $\bxi^k_n$ such that its normal component vanishes, or equivalently, 
choose $\bxi^k_n$ to be the minimal-norm element of the second solution component among all 
solutions $(\bu^k_n,\bxi^k_n)$ of the problem \eqref{S1}--\eqref{S2}.
\end{remark}

Then, we turn to a derivation of upper bounds on the semi-discrete solutions.

\begin{lemma}\label{lem:4.4}
Assume \eqref{assump1}, $H(\psi)$ and \eqref{small}. Then, there is a constant $c>0$ such that
\begin{align}
& \max_{0\le n\le N}\|\bu^k_n\|_{\boldsymbol H}\le c, \label{C1} \\[0.2mm]
& \sum_{n=1}^N \|\bu^k_n-\bu^k_{n-1}\|_{\boldsymbol H}^2\le c, \label{C2} \\[0.2mm]
& k \sum_{n=1}^N \|\bu^k_n\|_{\boldsymbol V}^2\le c, \label{C3}\\[0.2mm]
& \|\bxi^k_n\|_{{\boldsymbol S}^*}\le c\left(1+\|\bu^k_n\|_{\boldsymbol V}\right). \label{S14a}
\end{align}
\end{lemma}
\begin{proof}
Apply \eqref{S1} to $\bv=\bu^k_n$:
\begin{equation}
\frac{1}{k}(\bu^k_n-\bu^k_{n-1},\bu^k_n)_{\boldsymbol H}+a(\bu^k_n,\bu^k_n)
+\langle\bxi^k_n,\gamma_{\rm div}\bu^k_n\rangle_{{\boldsymbol S}^*\times {\boldsymbol S}}
=\langle\fb^k_n,\bu^k_n\rangle_{\bVVdiv}.\label{5.11a}
\end{equation}
For the first term on the left hand side, 
\[(\bu^k_n-\bu^k_{n-1},\bu^k_n)_{\boldsymbol H}
=\|\bu^k_n\|_{\boldsymbol H}^2-\|\bu^k_{n-1}\|_{\boldsymbol H}^2+\|\bu^k_n-\bu^k_{n-1}\|_{\boldsymbol H}^2.\]
For the third term on the left hand side, using \eqref{S2}, \eqref{3.27a} and $H(\psi)$,
\begin{align*}
\langle\bxi^k_n,\gamma_{\rm div}\bu^k_n\rangle_{{\boldsymbol S}^*\times {\boldsymbol S}} 
& = - \langle\bxi^k_n,-\gamma_{\rm div}\bu^k_n\rangle_{{\boldsymbol S}^*\times {\boldsymbol S}}\\
& \ge -\Psi^0(\gamma_{\rm div}\bu^k_n;-\gamma_{\rm div}\bu^k_n)\\
&=-\left[\Psi^0(\gamma_{\rm div}\bu^k_n;-\gamma_{\rm div}\bu^k_n)+\Psi^0(\bzero;\gamma_{\rm div}\bu^k_n)
\right]+\Psi^0(\bzero;\gamma_{\rm div}\bu^k_n)\\
&\ge -\alpha_\psi\lambda_\tau^{-1} \|\bu^k_n\|_{\boldsymbol V}^2-c_1-c_2\|\bu^k_n\|_{\boldsymbol V}.
\end{align*}
Then apply the modified Cauchy-Schwarz inequality \eqref{mCS} to bound the term $c_2\|\bu^k_n\|_{\boldsymbol V}$ 
to find that for any $\epsilon>0$, there is an $\epsilon$-dependent constant $c_3$ such that
\[ \langle\bxi^k_n,\gamma_{\rm div}\bu^k_n\rangle_{{\boldsymbol S}^*\times {\boldsymbol S}}
\ge -\left(\alpha_\psi\lambda_\tau^{-1}+\epsilon\right)\|\bu^k_n\|_{\boldsymbol V}^2-c_3.\]
For the right hand side of \eqref{5.11a}, 
\[ \langle\fb^k_n,\bu^k_n\rangle_{\bVVdiv}\le\|\fb^k_n\|_{{\boldsymbol V}^*}\|\bu^k_n\|_{{\boldsymbol V}}, \]
and then apply the modified Cauchy-Schwarz inequality to get an $\epsilon$-dependent constant $c$ such that
\[ \langle\fb^k_n,\bu^k_n\rangle_{\bVVdiv}\le \epsilon\,\|\bu^k_n\|_{{\boldsymbol V}}^2
+c\,\|\fb^k_n\|_{{\boldsymbol V}^*}^2.\]
Choose $\epsilon=m/4$ for $m$ defined in \eqref{small}.  Then, we derive from \eqref{5.11a} that
\[ \frac{1}{2\,k}\left(\|\bu^k_n\|_{\boldsymbol H}^2-\|\bu^k_{n-1}\|_{\boldsymbol H}^2
+\|\bu^k_n-\bu^k_{n-1}\|_{\boldsymbol H}^2\right) + \frac{m}{2}\,\|\bu^k_n\|_{\boldsymbol V}^2
\le \frac{c_4}{2}\left(1+\|\fb^k_n\|_{{\boldsymbol V}^*}^2\right),\]
or 
\[ \|\bu^k_n\|_{\boldsymbol H}^2-\|\bu^k_{n-1}\|_{\boldsymbol H}^2
+\|\bu^k_n-\bu^k_{n-1}\|_{\boldsymbol H}^2 + m\, k\,\|\bu^k_n\|_{\boldsymbol V}^2
\le c_4 k\left(1+\|\fb^k_n\|_{{\boldsymbol V}^*}^2\right). \]
Change $n$ to $j$ in the above inequality and add the inequality for $j$ from 1 to $n$:
\begin{align*}
& \|\bu^k_n\|_{\boldsymbol H}^2 +\sum_{j=1}^n \|\bu^k_j-\bu^k_{j-1}\|_{\boldsymbol H}^2
+m\, k\sum_{j=1}^n \|\bu^k_j\|_{\boldsymbol V}^2\\
&\qquad \le \|\bu_{k,0}\|_{\boldsymbol H}^2 +c_4 n\,k+c_4 k\sum_{j=1}^n \|\fb^k_j\|_{\boldsymbol V^*}^2.
\end{align*}
Recall that $\bu_{k,0}\to \bu_0$ in $\bH$, implying the uniform boundedness of $\|\bu_{k,0}\|_{\boldsymbol H}$
with respect to $k$.  Note that $n\,k\le T$ for any $1\le n\le N$.  By \eqref{5.1a}, the third term on the 
right side of the above inequality is also uniformly bounded with respect to $k$.
Hence, \eqref{C1}, \eqref{C2} and \eqref{C3} hold.

From \eqref{S2} and assumptions on $\Psi$, 
\[ \|\bxi^k_n\|_{{\boldsymbol S}^*}\le c\left(1+\|\gamma_{\rm div}\bu^k_n\|_{\boldsymbol S}\right)
\le c\left(1+\|\bu^k_n\|_{\boldsymbol V}\right).\]
Hence, \eqref{S14a} holds. \hfill
\end{proof}

We turn to a convergence analysis of the Rothe method.  Define a continuous piecewise linear function
$\bu_k\colon [0,T]\to \bV_{\rm div}$ and three piecewise constant functions $\overline{\bu}_k\colon [0,T]\to \bV_{\rm div}$, 
$\overline{\bxi}_k\colon [0,T]\to \bS^*$, and $\overline{\fb}_k\colon [0,T]\to {\boldsymbol V}^*$ by
\[ \bu_k(t)=\bu^k_n+\left(\frac{t}{k}-n\right)(\bu^k_n-\bu^k_{n-1}),\quad 
\overline{\bu}_k(t)=\bu^k_n, \quad \overline{\bxi}_k(t)=\bxi^k_n, \quad \overline{\fb}_k(t)=\fb^k_n \]
for $t\in (t_{n-1},t_n]$ if $2\le n\le N$, and for $t\in [0,t_1]$ if $n=1$.  
By \eqref{5.1a}, we have 
\begin{equation}
\|\overline{\fb}_k\|_{{\cal V}^*} \le \|\fb\|_{{\cal V}^*}. 
\label{5.12a}
\end{equation}
Also, the following convergence holds (\cite{CG99}):
\begin{equation}
\overline{\fb}_k\to \fb \quad{\rm in}\ {\cal V}^*. 
\label{5.12b}
\end{equation}

Rewrite \eqref{S1} and \eqref{S2} as, a.e.\ $t\in (0,T)$,
\begin{align}
& \bu_k^\prime(t) + A\overline{\bu}_k(t)+\gamma_{\rm div}^*\overline{\bxi}_k(t)=\overline{\fb}_k(t), \label{S20a}\\[0.5mm]
& \overline{\bxi}_k(t)\in \partial\Psi(\gamma_{\rm div}\overline{\bu}_k(t)). \label{S20b}
\end{align}
Define Nemytskii operators $\cA\colon\cV_{\rm div}\to \cV^*_{\rm div}$ and $\tilde{\gamma}\colon\cV_{\rm div}\to \cS$ by 
\begin{align*}
(\cA\bv)(t) &=A\bv(t),\\
(\tilde{\gamma}\bv)(t) &=\gamma_{\rm div} \bv(t). 
\end{align*}
Then, $\cA\in {\cal L}(\cV_{\rm div};\cV_{\rm div}^*)$, $\tilde{\gamma}\in {\cal L}(\cV_{\rm div};\cS)$. We derive from \eqref{S20a}--\eqref{S20b} that for a.e.\ $t\in(0,T)$,
\begin{align*}
& (\bu_k^\prime(t),\bv)_{\boldsymbol H} +\langle A\overline{\bu}_k(t),\bv\rangle_{\bVVdiv}
+\langle \overline{\bxi}_k(t),\gamma_{\rm div}\bv\rangle_{{\boldsymbol S}^*\times {\boldsymbol S}}= \langle \overline{\fb}_k(t),\bv\rangle_{\bVVdiv}\quad \forall\,\bv\in \bV_{\rm div}, \\[0.5mm]
& \overline{\bxi}_k(t)\in \partial\Psi(\gamma_{\rm div}\overline{\bu}_k(t)),
\end{align*}
or
\begin{align}
& (\bu_k^\prime,\bv)_{\cal H} +\langle \cA\overline{\bu}_k,\bv\rangle_{\cVVdiv}
+\langle \overline{\bxi}_k,\tilde{\gamma}\bv\rangle_{{\cal S}^*\times {\cal S}}= \langle \overline{\fb}_k,\bv\rangle_{\cVVdiv}\quad \forall\,\bv\in \cV_{\rm div}, \label{S21a}\\[0.5mm]
& \overline{\bxi}_k(t)\in \partial\Psi(\gamma_{\rm div}\overline{\bu}_k(t)),\quad {\rm a.e.}\ t\in(0,T). \label{S21b}
\end{align}

Let us establish some bounds on $\{\overline{\bu}_k\}$, $\{\bu_k\}$, and $\{\overline{\bxi}_k\}$ next.

\begin{lemma}\label{lem:bd}
Under the assumptions \eqref{assump1}, $H(\psi)$ and \eqref{small}, there exists a constant $c>0$ such that
\begin{align}
& \|\overline{\bu}_k\|_{L^\infty(0,T;{\boldsymbol H})}\le c, \label{S22a}\\
& \|\overline{\bu}_k\|_{\cal V}\le c, \label{S22b}\\
& \|\overline{\bu}_k\|_{M^{2,2}(0,T;{\boldsymbol V},{\boldsymbol V}^*)}\le c, \label{S22c}\\
& \|\bu_k\|_{C([0,T];{\boldsymbol H})}\le c, \label{S23a}\\
& \|\bu_k\|_{\cal V}\le c, \label{S23b}\\
& \|\bu_k^\prime\|_{{\cal V}^*}\le c, \label{S23c}\\
& \|\overline{\bxi}_k\|_{{\cal S}^*}\le c. \label{S24}
\end{align}
\end{lemma}
\begin{proof}
The bound \eqref{S22a} is a direct consequence of \eqref{C1}.  The bound \eqref{S22b} follows from \eqref{C3}:
\[ \|\overline{\bu}_k\|_{\cal V}^2 = \sum_{n=1}^N \|\bu_n^k\|_{\boldsymbol V}^2 k \le c. \]

For any $t\in (t_{n-1},t_n]$ if $2\le n\le N$, or for any $t\in[0,t_1]$ if $n=1$, by \eqref{C1},
\[ \|\bu_k(t)\|_{\boldsymbol H} \le \max\left\{ \|\bu^k_n\|_{\boldsymbol H}, \|\bu^k_{n-1}\|_{\boldsymbol H}\right\}\le c. \]
Notice that $\bu_k(t)$ is continuous in $t$.  So \eqref{S23a} holds.

The bound \eqref{S23b} is established from
\[  \|\bu_k\|_{\cal V}^2=\int_0^T \|\bu_k(t)\|_{\boldsymbol V}^2 dt=\sum_{n=1}^N\int_{t_{n-1}}^{t_n}
\left\|\bu^k_n+\left(\frac{t}{k}-n\right)(\bu^k_n-\bu^k_{n-1})\right\|_{\boldsymbol V}^2dt
\le c \sum_{n=0}^N k\,\|\bu^k_n\|_{\boldsymbol V}^2 \]
which is uniformly bounded with respect to $k$ by \eqref{C3}. 

By \eqref{S14a},
\[ \|\overline{\bxi}_k\|_{{\boldsymbol S}^*}^2 =\sum_{n=1}^N k\,\|\bxi^k_n\|_{{\boldsymbol S}^*}^2 
\le c\,k \sum_{n=1}^N \left(1+\|\bu^k_n\|_{\boldsymbol V}^2\right)
\le c+c\,k \sum_{n=1}^N \|\bu^k_n\|_{\boldsymbol V}^2.\]
Then, applying \eqref{C3}, we derive \eqref{S24}.

From \eqref{S20a},
\[ \bu_k^\prime(t)= \overline{\fb}_k(t)-A\overline{\bu}_k(t)-\gamma_{\rm div}^*\overline{\bxi}_k(t)
\quad {\rm a.e.}\ t\in(0,T).\]
Then,
\[ \|\bu_k^\prime(t)\|_{{\boldsymbol V}^*}\le \|\overline{\fb}_k(t)\|_{{\boldsymbol V}^*} 
+ \|A\|\,\|\overline{\bu}_k(t)\|_{\boldsymbol V}
+ \|\gamma_{\rm div}^*\| \|\overline{\bxi}_k(t)\|_{{\boldsymbol S}^*}^2\quad {\rm a.e.}\ t\in(0,T).\]
Consequently, for some constant $\overline{c}>0$, 
\[\|\bu_k^\prime\|_{{\cal V}^*}\le \overline{c} \left[ \|\overline{\fb}_k\|_{{\cal V}^*}
+\|\overline{\bu}_k\|_{\cal V}+\|\overline{\bxi}_k\|_{{\cal S}^*}\right],\]
which is bounded by a constant $c$ due to \eqref{5.12a}, \eqref{S22b} and \eqref{S24}.  Thus, \eqref{S23c} holds.

Finally, we derive the bound \eqref{S22c}.  By definition,
\[ \|\overline{\bu}_k\|_{BV^2(0,T;{\boldsymbol V}^*)}^2=\sup_{0=s_0<s_1<\cdots<s_{i_0}=T} 
\sum_{i=1}^{i_0} \|\overline{\bu}_k(s_i)-\overline{\bu}_k(s_{i-1})\|_{{\boldsymbol V}^*}^2, \]
where the supremum is taken with respect to all partitions of $[0,T]$.  For a particular partition
$0=s_0<s_1<\cdots<s_{i_0}=T$, we have positive integers $m_1\le m_2\le\cdots\le m_{i_0}=N$ such that
$s_i\in \left((m_i-1)\,k,m_i k\right]$, $1\le i\le i_0$.  Then,
\[ \overline{\bu}_k(s_i)-\overline{\bu}_k(s_{i-1})=\bu^k_{m_i}-\bu^k_{m_{i-1}}. \]
By the triangle inequality,
\[ \|\overline{\bu}_k(s_i)-\overline{\bu}_k(s_{i-1})\|_{{\boldsymbol V}^*}=\|\bu^k_{m_i}-\bu^k_{m_{i-1}}\|_{{\boldsymbol V}^*}\le\sum_{j=m_{i-1}+1}^{m_i}\|\bu^k_j-\bu^k_{j-1}\|_{{\boldsymbol V}^*}\]
and so,
\[ \|\overline{\bu}_k(s_i)-\overline{\bu}_k(s_{i-1})\|_{{\boldsymbol V}^*}^2
\le \left(m_i-m_{i-1}\right) \sum_{j=m_{i-1}+1}^{m_i}\|\bu^k_j-\bu^k_{j-1}\|_{{\boldsymbol V}^*}^2
\le \left(m_i-m_{i-1}\right) \sum_{j=1}^N\|\bu^k_j-\bu^k_{j-1}\|_{{\boldsymbol V}^*}^2.\]
As a result,
\begin{align*}
\sum_{i=1}^{i_0} \|\overline{\bu}_k(s_i)-\overline{\bu}_k(s_{i-1})\|_{{\boldsymbol V}^*}^2
& \le \sum_{i=1}^{i_0}\left(m_i-m_{i-1}\right)\sum_{j=1}^N\|\bu^k_j-\bu^k_{j-1}\|_{{\boldsymbol V}^*}^2\\
& = N \sum_{j=1}^N\|\bu^k_j-\bu^k_{j-1}\|_{{\boldsymbol V}^*}^2\\
& =T\,k\sum_{j=1}^N\|\bu^k_j-\bu^k_{j-1}\|_{{\boldsymbol V}^*}^2\\
&=T\,\|\bu_k^\prime\|_{{\cal V}^*}^2
\end{align*}
which is uniformly bounded with respect to $k$ by \eqref{S23c}.  Thus, 
\[ \|\overline{\bu}_k\|_{M^{2,2}(0,T;{\boldsymbol V},{\boldsymbol V}^*)}
=\|\overline{\bu}_k\|_{\cal V}+\|\overline{\bu}_k\|_{BV^2(0,T;{\boldsymbol V}^*)}\le c,\]
i.e., \eqref{S22c} holds. \hfill
\end{proof}

\begin{lemma}\label{lem:main}
Assume \eqref{assump1}, $H(\psi)$ and \eqref{small}. For a subsequence $\{k\}$, we have, as $k\to 0$,
\begin{align}
& \overline{\bu}_k \rightharpoonup \bu\quad{\rm in}\ \cV_{\rm div}, \label{S25}\\
& \overline{\bu}_k \to \bu\quad{\rm in}\ \cH, \label{S26}\\
& \overline{\bu}_{k,\tau} \to \bu_\tau\quad{\rm in}\ \cS \label{S27}\\
& \bu_k \rightharpoonup \bu\quad{\rm in}\ \cV_{\rm div}, \label{S28}\\
& \bu^\prime_k \rightharpoonup \bu^\prime\quad{\rm in}\ \cV^*, \label{S29}\\
& \overline{\bxi}_k\rightharpoonup \bxi\quad{\rm in}\ \cS^*, \label{S30}
\end{align}
and $(\bu,\bxi)$ solves Problem \ref{p3.10div}.
\end{lemma}
\begin{proof}
From the bounds in Lemma \ref{lem:bd}, we can assume that, passing to a subsequence if necessary, there exist
$\overline{\bu}\in\mathcal{V}_{\rm div}\cap L^\infty(0,T;\bH)$, $\bu\in\mathcal{V}_{\rm div}\cap L^\infty(0,T;\bH)$,
$\tilde{\bu}\in\mathcal{V}^*$, and $\bxi\in\mathcal{S}^*$ such that as $k\to 0$,
\begin{align}
& \overline{\bu}_k\weak \overline{\bu}\ {\rm in}\ \mathcal{V}\ {\rm and}\
\overline{\bu}_k\rightharpoonup^* \overline{\bu}\ {\rm in}\  L^\infty(0,T;\bH),\label{hvcon01}\\
& \bu_k\weak \bu \ {\rm in}\ \mathcal{V}\ {\rm and}\
 \bu_k\rightharpoonup^* \bu\ {\rm in}\ L^\infty(0,T;\bH),\label{hvcon02}\\
& \bu_k'\weak \tilde{\bu}\ {\rm in}\ \mathcal{V}^*,\label{hvcon03}\\
& \overline{\bxi}_k\weak \bxi\ {\rm in}\ \mathcal{S}^*.\label{hvcon04}
\end{align}
First we show that $\overline{\bu}=\bu$. From
\[ \|\overline{\bu}_k-\bu_k\|_{\mathcal{V}^*}^2=\sum_{i=1}^N\int_{(i-1)k}^{ik}(ik-t)^2 
\Big\|\frac{\bu_{k,i}-\bu_{k,i-1}}{k}\Big\|_{{\boldsymbol V}^*}^2dt 
= \frac{k^2}{3}\|\bu_k'\|_{\mathcal{V}^*}^2, \]
we know that $\overline{\bu}_k-\bu_k\rightarrow 0$ in $\mathcal{V}^*$ as $k\rightarrow0$. On the other hand, 
by \eqref{hvcon01} and \eqref{hvcon02}, $\overline{\bu}_k-\bu_k\weak \overline{\bu}-\bu$ in $\mathcal{V}$.
Since the embedding $\mathcal{V}\subset \mathcal{V}^*$ is continuous, we also have
$\overline{\bu}_k-\bu_k\weak \overline{\bu}-\bu$ in $\mathcal{V}^*$. Hence,
$\overline{\bu}-\bu=0$, i.e.\ $\overline{\bu}=\bu$. Since $\bu_k\weak \bu$ in $\mathcal{V}$
and $\bu_k'\weak \tilde{\bu}$ in $\mathcal{V}^*$, we conclude (cf.\ \cite[Proposition 23.19]{ZeiIIA})
that $\tilde{\bu}=\bu'$.  Thus, for all $\bv\in \mathcal{V}_{\rm div}$, we obtain
\begin{equation}
(\bu_k',\bv)_{\mathcal{H}}=\langle \bu_k',\bv\rangle_{\cVVdiv}\rightarrow \langle \bu',\bv\rangle_{\cVVdiv}.
\label{hvsov01}
\end{equation}
Since $\mathcal{A}\colon \mathcal{V}_{\rm div}\to \mathcal{V}^*_{\rm div}$ is linear and continuous, 
it is weakly continuous. Thus, $\overline{\bu}_k\weak \bu$ in $\mathcal{V}_{\rm div}$ implies that
\begin{equation}
\langle \mathcal{A}\overline{\bu}_k,\bv\rangle_{\cVVdiv}\rightarrow\langle \mathcal{A}\bu,\bv\rangle_{\cVVdiv}.\label{hvsov02}
\end{equation}
From \eqref{hvcon04} we get
\begin{equation}
\langle\overline{\bxi}_k,\tilde{\gamma}\bv\rangle_{\mathcal{S}^*\times\mathcal{S}}\rightarrow
\langle \bxi,\tilde{\gamma}\bv\rangle_{\mathcal{S}^*\times\mathcal{S}}.\label{hvsov04}
\end{equation}
It follows from \eqref{5.12b} that $\overline{\fb}_k\rightarrow \fb$ in $\mathcal{V}^*$.  Thus,
\begin{equation}
\langle \fb_k,\bv\rangle_{\cVVdiv}\rightarrow\langle \fb,\bv\rangle_{\cVVdiv}.\label{hvsov05}
\end{equation}
Using \eqref{hvsov01}--\eqref{hvsov05}, we can pass to the limit $k\to 0$ in \eqref{S21a} and obtain
\begin{equation}
\langle \bu',\bv\rangle_{\cVVdiv}+\langle \mathcal{A}\bu,\bv\rangle_{\cVVdiv}
+\langle \bxi,\tilde{\gamma} \bv\rangle_{\mathcal{S}^*\times\mathcal{S}}
= \langle \fb,\bv\rangle_{\cVVdiv}\quad\forall\,\bv\in \mathcal{V}_{\rm div}.
\label{hvsov06}
\end{equation}
From this, we deduce \eqref{3.54} in $\bV^*_{\rm div}$ for a.e.\ $t\in (0,T)$.  

Recall the bound \eqref{S22c}. Applying Theorem \ref{th:lb} and Sobolev embedding theorems, we know that, 
up to a further subsequence if necessary, $\overline{\bu}_k\to \bu$ in $L^2(0,T;H^\delta(\Omega)^d)$ 
for some number $\delta\in (1/2,1)$.  Since $\gamma\colon H^\delta(\Omega)^d \to\bS$ is compact, 
for a further subsequence if needed, we have 
$\gamma_{\rm div}\overline{\bu}_k(t)\rightarrow\gamma_{\rm div}\bu(t)$ in $\bS$ for a.e.\ $t\in (0,T)$. 
Since $\partial \psi: \bS\rightarrow 2^{{\boldsymbol S}^*}$ has nonempty, closed and convex values and is 
upper semicontinuous from $\bS$ endowed with strong topology into $\bS^*$ endowed with weak topology
(cf.\ \cite[Proposition 5.6.10]{DMP20031}), we deduce from \eqref{hvcon04}, \eqref{S21b} and Theorem \ref{th:sv} that
\begin{equation}
\bxi(t)\in\partial \psi(\gamma_{\rm div} \bu(t))\quad{\rm a.e.}\ t\in (0,T).\label{hvsov07}
\end{equation}

Finally, we pass to the limit with the initial conditions on the function $\bu_k$.
Since $\bu_k\weak \bu$ in $\mathcal{V}$, $\bu_k'\weak \bu'$ in $\mathcal{V}^*$, we have 
$\bu_k(t)\weak \bu(t)$ in $\bH$ for all $t\in [0,T]$ (cf.\ \cite[Lemma 4\,(b)]{MO09}).  Therefore,
$\bu_{k,0}=\bu_k(0)\weak \bu(0)$ in $\bH$. Since $\bu_{k,0}\rightarrow \bu_0$ in $\bH$, we have $\bu(0)=\bu_0$.  

In conclusion, $(\bu,\bxi)$ is a solution of Problem \ref{p3.10div}.  \hfill
\end{proof}

Summarizing, we have the next two results.

\begin{theorem}\label{thm:main}
Assume \eqref{assump1}, $H(\psi)$ and \eqref{small}. Then, Problem \ref{p3.10div} has a solution $(\bu,\bxi)\in \cW_{\rm div}\times \cS^*$ and the solution component $\bu$ is unique.
\end{theorem}
\begin{proof}
Existence of a solution to Problem \ref{p3.10div} follows from Lemma \ref{lem:main}. 

Let us show the uniqueness of the solution component $\bu$.  Suppose both $(\bu_1,\bxi_1),(\bu_2,\bxi_2)
\in\cW_{\rm div}\times\cS^*$ are solutions of Problem \ref{p3.10div}.  Then for a.e.\ $t\in(0,T)$,
\begin{align}
& \langle \bu_1^\prime(t),\bv\rangle_{\bVVdiv}+a(\bu_1(t),\bv)+\langle\bxi_1(t),\gamma_{\rm div}\bv\rangle_{{\boldsymbol S}^*\times {\boldsymbol S}}
=\langle \fb(t),\bv\rangle_{\bVVdiv}\quad \forall\,\bv\in \bV_{\rm div}, \label{4.51}\\[0.5mm]
& \bxi_1(t)\in \partial\Psi(\gamma_{\rm div}\bu_1(t)), \label{4.52}\\[0.2mm]
& \langle \bu_2^\prime(t),\bv\rangle_{\bVVdiv}+a(\bu_2(t),\bv)+\langle\bxi_2(t),\gamma_{\rm div}\bv\rangle_{{\boldsymbol S}^*\times {\boldsymbol S}}
=\langle \fb(t),\bv\rangle_{\bVVdiv}\quad \forall\,\bv\in \bV_{\rm div}, \label{4.53}\\[0.5mm]
& \bxi_2(t)\in \partial\Psi(\gamma_{\rm div}\bu_2(t)), \label{4.54}
\end{align}
and 
\begin{equation}
\bu_1(0)=\bu_2(0)=\bu_0.  \label{4.55}
\end{equation}
Take $\bv=\bu_2(t)-\bu_1(t)$ in \eqref{4.51}, $\bv=\bu_1(t)-\bu_2(t)$ in \eqref{4.53}, and add the two 
inequalities to obtain
\begin{align}
& \frac{1}{2}\,\frac{d}{dt}\|\bu_1(t)-\bu_2(t)\|_{\boldsymbol H}^2+a(\bu_1(t)-\bu_2(t),\bu_1(t)-\bu_2(t))\nonumber\\
&{}\qquad +\langle \bxi_1(t)-\bxi_2(t),\gamma_{\rm div}\bu_1(t)-\gamma_{\rm div}\bu_2(t)\rangle_{{\boldsymbol S}^*\times {\boldsymbol S}}= 0
\label{5.43a}
\end{align}
for a.e.\ $t\in (0,T)$.  Then, with $m=2\,\mu-\alpha_\psi\lambda_\tau^{-1}$ defined in \eqref{small},
\begin{equation}
\frac{1}{2}\,\frac{d}{dt}\|\bu_1(t)-\bu_2(t)\|_{\boldsymbol H}^2+m\,\|\bu_1(t)-\bu_2(t)\|_{\boldsymbol V}^2\le 0,
\quad {\rm a.e.}\ t\in (0,T).
\label{5.43b}
\end{equation}
Integrate the inequality from $0$ to $t\in(0,T)$ to obtain
\[ \frac{1}{2}\,\|\bu_1(t)-\bu_2(t)\|_{\boldsymbol H}^2+
m\int_0^t\|\bu_1(s)-\bu_2(s)\|_{\boldsymbol V}^2ds\le 0,\quad t\in (0,T). \]
Since $m>0$ by assumption, we deduce from the above inequality that
\[ \bu_1(t)-\bu_2(t)=\bzero\ {\rm for}\ t\in(0,T).\]
Therefore, we have the uniqueness of $\bu$. \hfill
\end{proof}

\begin{theorem}\label{thm:main1}
Assume \eqref{assump1}, $H(\psi)$ and \eqref{small}. Then, Problem \ref{p3.5div} has a unique solution $\bu\in \cW_{\rm div}$.
\end{theorem}
\begin{proof}
Theorem \ref{thm:main} guarantees the existence of a solution $\bu\in \cW_{\rm div}$ for Problem \ref{p3.5div}.

For the uniqueness, assume $\bu_1,\bu_2\in \cW_{\rm div}$ are two solutions of Problem \ref{p3.5div}.
Then, similar to \eqref{5.43a}, we have 
\begin{align*}
& \frac{1}{2}\,\frac{d}{dt}\|\bu_1(t)-\bu_2(t)\|_{\boldsymbol H}^2+a(\bu_1(t)-\bu_2(t),\bu_1(t)-\bu_2(t))\\
&{}\qquad \le \Psi^0(\gamma_{\rm div}\bu_1(t);\gamma_{\rm div}\bu_2(t)-\gamma_{\rm div}\bu_1(t))
+ \Psi^0(\gamma_{\rm div}\bu_2(t);\gamma_{\rm div}\bu_1(t)-\gamma_{\rm div}\bu_2(t))
\end{align*}
for a.e.\ $t\in (0,T)$.  We again have the inequality \eqref{5.43b} from which, the uniqueness follows. \hfill
\end{proof}

\section{Existence and uniqueness of the pressure field}\label{sec:p}

We have already shown in Theorem \ref{thm:main1} that Problem \ref{p3.5div} has a unique solution 
$\bu\in \cW_{\rm div}$.  Since
\[ \Psi^0(\gamma_{\rm div}\bu(t);\gamma_{\rm div}\bv)
\le I_{\Gamma_S}(\psi^0(\gamma_{\rm div}\bu(t);\gamma_{\rm div}\bv)),\]
we apply Theorem \ref{thm:main1} to know that there is an element $\bu\in \cW_{\rm div}$ such that for 
a.e.\ $t\in (0,T)$,
\[ \langle \bu^\prime(t),\bv\rangle_{\bVVdiv}+a(\bu(t),\bv)
+I_{\Gamma_S}(\psi^0(\gamma_{\rm div}\bu(t);\gamma_{\rm div}\bv))
\ge \langle \fb(t),\bv\rangle_{\bVVdiv}\quad \forall\,\bv\in \bV_{\rm div}.\]
Moreover, similar to the proof of Theorem \ref{thm:main1}, we can show that the element $\bu\in \cW_{\rm div}$ is unique.

From Lemma \ref{lem:main}, we know that $\bu^\prime\in \cV^*$.
Moreover, $\cA \bu$ is an element in $\cV^*$ since $\bu\in \cV$.  Also, for a.e.\ $t\in [0,T]$,
$\gamma^*_{\rm div}\bxi(t)$ can be viewed an element in $\bV^*$ through the formula
\[ \langle \gamma^*_{\rm div}\bxi(t), \bv\rangle=\langle \bxi(t), \gamma\bv\rangle\quad\forall\,\bv\in \bV. \]
In addition, by assumption \eqref{assump1}, $\fb(t)\in \bV^*$ for a.e.\ $t\in (0,T)$.  Therefore,
there exists a unique solution $\bu\in \cW_{\rm div}$ to the following problem.

\begin{problem}\label{p3.5div1}
Find $\bu\in \cW_{\rm div}$ such that for a.e.\ $t\in (0,T)$,
\begin{align}
\langle \bu^\prime(t),\bv\rangle_{\bVV}+a(\bu(t),\bv)+I_{\Gamma_S}(\psi^0(\gamma\bu(t);\gamma\bv))
\ge \langle \fb(t),\bv\rangle_{\bVV}\quad \forall\,\bv\in \bV_{\rm div}, \label{5.51}
\end{align}
and 
\begin{equation}
\bu(0)=\bu_0.  \label{5.51a}
\end{equation}
\end{problem}

We deduce from \eqref{5.51} that
\begin{equation}
\langle \bu^\prime(t),\bv\rangle_{\bVV}+a(\bu(t),\bv)= \langle \fb(t),\bv\rangle_{\bVV}\quad 
\forall\,\bv\in \bV_{\rm 0, div}:=\bV_0\cap \bV_{\rm div}. \label{5.52}
\end{equation}
Thanks to the inf-sup property \eqref{inf-sup}, there exists a function $p(t)\in Q$ such that
\begin{equation}
\langle \bu^\prime(t),\bv\rangle_{\bVV}+a(\bu(t),\bv)-b(\bv,p(t))= \langle \fb(t),\bv\rangle_{\bVV}
\quad \forall\,\bv\in \bV_0. \label{5.53}
\end{equation}
Thus,
\[ b(\bv,p(t))=\langle \bu^\prime(t),\bv\rangle_{\bVV}+a(\bu(t),\bv)-\langle \fb(t),\bv\rangle_{\bVV}
\quad \forall\,\bv\in \bV_0,\]
leading to
\[ b(\bv,p(t))\le \left(\|\bu^\prime(t)\|_{\boldsymbol V^*}+2\,\mu\,\|\bu(t)\|_{\boldsymbol V}
+\|\fb(t)\|_{\boldsymbol V^*}\right)\|\bv\|_{\boldsymbol V}\quad \forall\,\bv\in \bV_0.\]
Hence, with an application of the inf-sup property \eqref{inf-sup},
\[ \alpha_b \|p(t)\|_Q\le \|\bu^\prime(t)\|_{\boldsymbol V^*}+2\,\mu\,\|\bu(t)\|_{\boldsymbol V}
+\|\fb(t)\|_{\boldsymbol V^*}. \]
Consequently,
\[ p\in \cQ.\]

For an arbitrary $\bv\in \bV$, by the inf-sup property, there exists a function $\bv_1\in\bV_0$ such that
\begin{equation}
b(\bv_1,q)=b(\bv,q)\quad\forall\,q\in Q. \label{5.54}
\end{equation}
Define 
\[ \bv_2=\bv-\bv_1. \]
Then $\bv_2\in\bV$ and 
\[ b(\bv_2,q)=0\quad\forall\,q\in Q. \]
Hence, $\bv_2\in \bV_{\rm div}$.  By \eqref{5.51},
\[ \langle \bu^\prime(t),\bv_2\rangle_{\bVV}+a(\bu(t),\bv_2)+I_{\Gamma_S}(\psi^0(\gamma\bu(t);\gamma\bv_2))
\ge \langle \fb(t),\bv_2\rangle_{\bVV},\]
i.e.,
\[ \langle \bu^\prime(t),\bv-\bv_1\rangle_{\bVV}+a(\bu(t),\bv-\bv_1)
+I_{\Gamma_S}(\psi^0(\gamma\bu(t);\gamma(\bv-\bv_1)))\ge \langle \fb(t),\bv-\bv_1\rangle_{\bVV}.\]
Note that $\bv_1\in\bV_0$ and so $\gamma\bv_1=\bzero$ on $\Gamma_S$.  As a result,
\[ \langle \bu^\prime(t),\bv-\bv_1\rangle_{\bVV}+a(\bu(t),\bv-\bv_1)+I_{\Gamma_S}(\psi^0(\gamma\bu(t);\gamma\bv))
\ge \langle \fb(t),\bv-\bv_1\rangle_{\bVV}.\]
By \eqref{5.53},
\[ \langle \bu^\prime(t),\bv_1\rangle_{\bVV}+a(\bu(t),\bv_1)-b(\bv_1,p(t))= \langle \fb(t),\bv_1\rangle_{\bVV}.\]
By \eqref{5.54},
\[ b(\bv_1,p(t))=b(\bv,p(t)). \]
Thus,
\[ \langle \bu^\prime(t),\bv\rangle_{\bVV}+a(\bu(t),\bv)-b(\bv,p(t))+I_{\Gamma_S}(\psi^0(\gamma\bu(t);\gamma\bv))
\ge \langle \fb(t),\bv\rangle_{\bVV}.\]
Therefore, $(\bu,p)$ is a solution of Problem \ref{p1}.

We already know that the solution component $\bu$ is unique.  To show the uniqueness of $p$, assume
$(\bu,p_1)$ and $(\bu,p_2)$ are two solutions of Problem \ref{p1}. By \eqref{5.53}, for all $\bv\in \bV_0$,
\begin{align}
\langle \bu^\prime(t),\bv\rangle_{\bVV}+a(\bu(t),\bv)-b(\bv,p_1(t))= \langle \fb(t),\bv\rangle_{\bVV}, \label{5.55}\\
\langle \bu^\prime(t),\bv\rangle_{\bVV}+a(\bu(t),\bv)-b(\bv,p_2(t))= \langle \fb(t),\bv\rangle_{\bVV}. \label{5.56}
\end{align}
Subtract \eqref{5.56} from \eqref{5.55} to get
\[ b(\bv,p_1(t)-p_2(t))=0\quad \forall\,\bv\in \bV_0.\]
Applying the inf-sup property \eqref{inf-sup}, we have
\[ \alpha_b\|p_1(t)-p_2(t)\|_Q\le\sup_{\boldsymbol v\in {\boldsymbol V}_0}\frac{b(\bv,p_1(t)-p_2(t))}{\|\bv\|_{\boldsymbol V}}=0. \]
Hence, $p_1(t)-p_2(t)=0$ and the solution component $p$ is unique.

In conclusion, we have proved the following existence and uniqueness result.

\begin{theorem}\label{thm:main2}
Assume \eqref{assump1}, $H(\psi)$ and \eqref{small}. Then, Problem \ref{p1} has a unique solution 
$(\bu,p)\in \cW\times \cQ$.
\end{theorem}

\section{Lipschitz continuous dependence on data}\label{sec:c}

In this section, we explore the Lipschitz continuous dependence of the solution $\bu$ on the source function
$\fb$ and the initial value $\bu_0$.

\begin{theorem}
Assume \eqref{assump1}, $H(\psi)$ and \eqref{small}, and let $(\bu,p)\in \cW\times \cQ$ be the solution of 
Problem \ref{p1}.  Then, the mapping $(\fb,\bu_0)\mapsto \bu$ is Lipschitz continuous from $\cV^*\times \bH$
to $\cV\cap C(0,T;\bH)$.
\end{theorem}
\begin{proof}
For $i=1,2$, let $(\fb_i,\bu_{i,0})\in \cV^*\times \bH$ and let $(\bu_i,p_i)\in\cW\times \cQ$ be the corresponding
solution of Problem \ref{p1}.  From the definition of the problem, for $i=1,2$ and for a.e.\ $t\in (0,T)$,
\begin{align}
& \langle \bu_i^\prime(t),\bv\rangle_{\bVV}+a(\bu_i(t),\bv)-b(\bv,p_i(t))+I_{\Gamma_S}(\psi^0(\bu_{i,\tau}(t);\bv_\tau))
\ge \langle \fb_i(t),\bv\rangle_{\bVV}\quad \forall\,\bv\in \bV, \label{c1}\\[0.5mm]
& b(\bu_i(t),q)=0\quad\forall\,q\in Q, \label{c2}
\end{align}
and 
\begin{equation}
\bu_i(0)=\bu_{i,0}.  \label{c3}
\end{equation}
Take $\bv=\bu_2(t)-\bu_1(t)$ in \eqref{c1} for $i=1$, take $\bv=\bu_1(t)-\bu_2(t)$ in \eqref{c1} for $i=2$,
and add the two inequalities to obtain, for a.e.\ $t\in (0,T)$,
\begin{align*}
&\langle \bu_1^\prime(t)-\bu_2^\prime(t),\bu_1(t)-\bu_2(t)\rangle_{\bVV}+a(\bu_1(t)-\bu_2(t),\bu_1(t)-\bu_2(t))\\
&\qquad \le I_{\Gamma_S}(\psi^0(\bu_{1,\tau}(t), \bu_{2,\tau}(t)-\bu_{1,\tau}(t))
+ \psi^0(\bu_{2,\tau}(t), \bu_{1,\tau}(t)-\bu_{2,\tau}(t)) )\\
&\qquad \quad {}+ b(\bu_1(t)-\bu_2(t),p_1(t)-p_2(t)) + \langle \fb_1(t)-\fb_2(t),\bu_1(t)-\bu_2(t)\rangle_{\bVV}.
\end{align*}
It follows from \eqref{c2} that $b(\bu_1(t)-\bu_2(t),p_1(t)-p_2(t))=0$.  Note that
\begin{align*}
& \langle \bu_1^\prime(t)-\bu_2^\prime(t),\bu_1(t)-\bu_2(t)\rangle_{\bVV}
   = \frac{1}{2}\,\frac{d}{dt}\|\bu_1(t)-\bu_2(t)\|_{\boldsymbol H}^2,\\[0.2mm]
& a(\bu_1(t)-\bu_2(t),\bu_1(t)-\bu_2(t)) = 2\mu\,\|\bu_1(t)-\bu_2(t)\|_{\boldsymbol V}^2,\\
& I_{\Gamma_S}(\psi^0(\bu_{1,\tau}(t), \bu_{2,\tau}(t)-\bu_{1,\tau}(t))
   + \psi^0(\bu_{2,\tau}(t), \bu_{1,\tau}(t)-\bu_{2,\tau}(t)) )
  \le \alpha_\psi \lambda_\tau^{-1} \|\bu_1(t)-\bu_2(t)\|_{\boldsymbol V}^2,\\
& \langle \fb_1(t)-\fb_2(t),\bu_1(t)-\bu_2(t)\rangle_{\bVV}
   \le \|\fb_1(t)-\fb_2(t)\|_{{\boldsymbol V}^*} \|\bu_1(t)-\bu_2(t)\|_{\boldsymbol V}.
\end{align*}
Hence, 
\[ \frac{1}{2}\,\frac{d}{dt}\|\bu_1(t)-\bu_2(t)\|_{\boldsymbol H}^2+m\,\|\bu_1(t)-\bu_2(t)\|_{\boldsymbol V}^2 
\le \|\fb_1(t)-\fb_2(t)\|_{{\boldsymbol V}^*} \|\bu_1(t)-\bu_2(t)\|_{\boldsymbol V}, \]
where $m$ is defined in \eqref{small}.  Applying the inequality
\[ \|\fb_1(t)-\fb_2(t)\|_{{\boldsymbol V}^*} \|\bu_1(t)-\bu_2(t)\|_{\boldsymbol V}\le
\frac{m}{2}\,\|\bu_1(t)-\bu_2(t)\|_{\boldsymbol V}^2+\frac{1}{2\,m}\,\|\fb_1(t)-\fb_2(t)\|_{{\boldsymbol V}^*}^2,\]
we further derive that for a.e.\ $t\in (0,T)$,
\[ \frac{d}{dt}\|\bu_1(t)-\bu_2(t)\|_{\boldsymbol H}^2+m\,\|\bu_1(t)-\bu_2(t)\|_{\boldsymbol V}^2 
\le \frac{1}{m}\,\|\fb_1(t)-\fb_2(t)\|_{{\boldsymbol V}^*}^2.\]
Integrating the inequality from $0$ to $t$, we get that for any $t\in [0,T]$,
\begin{equation}
\|\bu_1(t)-\bu_2(t)\|_{\boldsymbol H}^2+m\int_0^t\|\bu_1(s)-\bu_2(s)\|_{\boldsymbol V}^2 ds
\le \|\bu_{1,0}-\bu_{2,0}\|_{\boldsymbol H}^2+\frac{1}{m}\int_0^t\|\fb_1(s)-\fb_2(s)\|_{{\boldsymbol V}^*}^2 ds.
\label{c5}
\end{equation}
The stated Lipschitz continuous dependence follows from the inequality \eqref{c5}. \hfill
\end{proof}

\section{Concluding remark}\label{sec:conc}

To the best of our knowledge, the reference \cite{HYZ26} is the first paper that provides an existence and uniqueness result for
both the velocity and pressure fields of a nonstationary Stokes hemivariational inequality for an incompressible 
fluid flow subject to a nonsmooth non-monotone boundary condition of friction type.  The well-posedness analysis
in \cite{HYZ26} is carried out under the following assumptions on the data
\begin{equation}
\bu_0\in\bV,\quad {\rm div}\bu_0=0\ {\rm in}\ \Omega,\quad \fb\in H^1(0,T;\bV^*)
\label{conc1}
\end{equation}
together with a compatibility condition on initial values of the data.  This paper provides a different approach
for the well-posedness analysis of the nonstationary Stokes hemivariational inequality so that the assumptions
\eqref{conc1} are replaced by the more natural ones:
\[ \bu_0\in\bH,\quad \fb\in L^2(0,T;\bV^*); \]
moreover, the compatibility condition on initial values of the data is dropped.

We also comment that the analysis in \cite{HYZ26} on a fully discrete scheme for solving the 
nonstationary Stokes hemivariational inequality is still valid without the compatibility condition 
on initial values of the data and the condition ${\rm div}\bu_0=0$ in $\Omega$.


\begin{thebibliography}{99}

\bibitem{AC1984}
J. P. Aubin and A. Cellina, \emph{Differential Inclusions. Set-Valued Maps and Viability Theory}. Springer,
Berlin, 1984.


\bibitem{BCKYZ15}
K. Bartosz, X. Cheng, P. Kalita, Y. Yu and C. Zheng, Rothe method for parabolic
variational-hemivariational inequalities, \emph{Journal of Mathematical Analysis and Applications} {\bf 423} (2015), 841--862.

\bibitem{CL2021}
S. Carl and V. K. Le, \emph{Multi-Valued Variational Inequalities and Inclusions}, Springer, New York, 2021.

\bibitem{CG99}
C. Carstensen and J. Gwinner, A theory of discretisation for nonlinear evolution inequalities
applied to parabolic Signorini problems, \emph{Annali di Matematica Pura ed Applicata}  {\bf 177} (1999), 363--394.

\bibitem{Cl1983}
F. H. Clarke, \emph{Optimization and Nonsmooth Analysis}, Wiley, Interscience, New York, 1983.

\bibitem{CLSW1998}
F. H. Clarke, Y.S. Ledyaev, R. J. Stern, and P. R. Wolenski, \emph{Nonsmooth Analysis and Control Theory},
Springer, New York, 1998.

\bibitem{DMP20031}
Z. Denkowski, S. Mig\'orski, and N. S. Papageorgiou,
\emph{An Introduction to Nonlinear Analysis: Theory}, Kluwer Academic, New York, 2003.

\bibitem{Dj14}
J. K. Djoko, On the time approximation of the Stokes equations with nonlinear slip boundary
conditions, \emph{International Journal of Numerical Analysis and Modeling, Series B}  {\bf 11} (2014), 34--53.

\bibitem{FCHCD20}
C. Fang, K. Czuprynski, W. Han, X. L. Cheng, and X. Dai, Finite element method for a stationary Stokes
hemivariational inequality with slip boundary condition, \emph{IMA Journal of Numerical Analysis} {\bf 40} (2020), 2696--2716.

\bibitem{Fu02}
H. Fujita, A coherent analysis of Stokes flows under boundary conditions of friction type,
\emph{Journal of Computational and Applied Mathematics}  {\bf 149} (2002), 57--69.

\bibitem{GR1986}
V. Girault and P. A. Raviart, \emph{Finite Element Methods for Navier-Stokes Equations:
Theory and Algorithms}, Springer-Verlag, Berlin, 1986.

\bibitem{Han20}
W. Han, Minimization principles for elliptic hemivariational inequalities,
\emph{Nonlinear Analysis: Real World Applications} {\bf 54} (2020), paper no.\ 103114.

\bibitem{Han21}
W. Han, A revisit of elliptic variational-hemivariational inequalities,
\emph{Numerical Functional Analysis and Optimization} {\bf 42} (2021), 371--395.

\bibitem{Han2024}
W. Han, \emph{An Introduction to Theory and Applications of Stationary Variational-Hemivariational Inequalities},
Springer, New York, 2024.

\bibitem{Han26}
W. Han, Variational-hemivariational inequalities: a brief survey on mathematical theory and 
numerical analysis, to appear in \emph{Approximation Theory and Special Functions}.

\bibitem{HFWH25}
W. Han, F. Feng, F. Wang, and J. Huang, Numerical analysis of hemivariational inequalities with 
applications in contact mechanics, \emph{Advances in Applied Mechanics} {\bf 60} (2025), 113--178.

\bibitem{HJY23}
W. Han, F. Jing, and Y. Yao, Stabilized mixed finite element methods for a Navier–Stokes hemivariational inequality, 
\emph{BIT Numerical Mathematics} {\bf 63} (2023), article number 46.

\bibitem{HS19AN}
W. Han and M. Sofonea, Numerical analysis of hemivariational inequalities in contact mechanics,
\emph{Acta Numerica} {\bf 28} (2019), 175--286.

\bibitem{HYZ26}
W. Han, Y. Yao, and S. Zeng, Well-posedness and numerical analysis of a nonstationary Stokes
hemivariational inequality, to appear in \emph{Mathematics and Mechanics of Solids}.

\bibitem{HMP1999}
J. Haslinger, M. Miettinen, and P. D. Panagiotopoulos, \emph{Finite Element Method for
Hemivariational Inequalities: Theory, Methods and Applications},
Kluwer Academic Publishers, Dordrecht, Boston, London, 1999.

\bibitem{Jo2016}
V. John, \emph{Finite Element Methods for Incompressible Flow Problems}, Springer, 2016.

\bibitem{Ka13}
P. Kalita, Convergence of Rothe scheme for hemivariational inequalities of parabolic type,
\emph{International Journal of Numerical Analysis and Modeling}  {\bf 10} (2013), 445--465.

\bibitem{LL08}
Y. Li and K. Li, Penalty finite element method for Stokes problem with nonlinear slip
boundary conditions, \emph{Applied Mathematics and Computation}  {\bf 204} (2008), 216--226.

\bibitem{LHZ22}
M. Ling, W. Han, and S. Zeng, A pressure projection stabilized mixed finite element
method for a Stokes hemivariational inequality, \emph{Journal of Scientific Computing} {\bf 92} (2022), article number 13.

\bibitem{MO09}
S. Mig\'orski and A. Ochal, Quasistatic hemivariational inequality via vanishing 
acceleration approach, \emph{SIAM Journal on Mathematical Analysis} {\bf 41} (2009), 1415--1435.

\bibitem{MOS2013}
S. Mig\'orski, A. Ochal, and M. Sofonea, \emph{Nonlinear Inclusions and Hemivariational Inequalities.
Models and Analysis of Contact Problems}, Advances in Mechanics and Mathematics 26, Springer, New York, 2013.

\bibitem{NP1995}
Z. Naniewicz and P. D. Panagiotopoulos, \emph{Mathematical Theory of Hemivariational
Inequalities and Applications}, Dekker, New York, 1995.

\bibitem{NH1981} J. Ne\v cas and I. Hlava\v cek, \emph{Mathematical Theory of
Elastic and Elastoplastic Bodies: An Introduction}, Elsevier, Amsterdam, 1981.

\bibitem{Pa83}
P. D. Panagiotopoulos, Nonconvex energy functions, hemivariational inequalities
and substationary principles, \emph{Acta Mechanica} {\bf 42} (1983), 160--183.

\bibitem{Rou2013}
T. Roubicek, \emph{Nonlinear Partial Differential Equations with Applications}, second edition,
Birkh\"{a}user Verlag, Basel, Boston, Berlin, 2013.

\bibitem{SM2025}
M. Sofonea and S. Mig\'orski, \emph{Variational-Hemivariational Inequalities with Applications},
second edition, CRC Press, Boca Raton, FL, 2025.

\bibitem{Te1979}
R. Temam, \emph{Navier-Stokes Equations: Theory and Numerical Analysis}, North-Holland, Amsterdam, 1979.

\bibitem{ZeiIIA}
E. Zeidler, \emph{Nonlinear Functional Analysis and Applications. II\,A}, Springer, New York, 1990.

\bibitem{ZMN22}
S. Zeng, S. Mig\'orski, and V. T. Nguyen, A class of hyperbolic variational-hemivariational inequalities
without damping terms, \emph{Advances in Nonlinear Analysis} {\bf 11} (2022), 1287--1306.

\end{thebibliography}
\end{document}